\documentclass[runningheads]{llncs}
\newcommand{\repeatthanks}{\textsuperscript{\thefootnote}}
\usepackage[T1]{fontenc}
\usepackage{graphicx}

\usepackage{amsmath, amssymb}
\usepackage{libertine}
\usepackage{adjustbox}
\usepackage{xcolor}
\usepackage{paralist}
\usepackage{enumitem}

\usepackage{multirow}

\definecolor{dartmouthgreen}{rgb}{0.05, 0.5, 0.06}
\usepackage{hyperref}
\hypersetup{
    colorlinks=true,
    linkcolor=blue,
    citecolor=dartmouthgreen,
    pdftitle={Probabilistic Lookahead Strong Branching},
    pdfauthor={Mexi, Shamsi, Besançon, Le Bodic}
}

\usepackage[capitalize,nameinlink]{cleveref} 

\usepackage[numbers]{natbib}

\newcommand{\rplus}[0]{\mathbb{R}_+}
\newcommand{\rpluspower}[1]{\mathbb{R}_+^#1}
\newcommand{\zplus}[0]{\mathbb{Z}_+}
\newcommand{\dmin}[1]{\ensuremath{d^*_{#1}}}

\newcommand{\newMVB}[1]{Pandora's MVB}

\usepackage{todonotes}

\usepackage{booktabs} 

\usepackage[ruled]{algorithm2e}

\usepackage[scaled=0.85]{FiraMono}

\DeclareMathOperator*{\argmax}{argmax}

\makeatletter
\renewcommand\bibsection%
{
  \section*{\refname
    \@mkboth{\MakeUppercase{\refname}}{\MakeUppercase{\refname}}}
}
\makeatother

\begin{document}

\title{Probabilistic Lookahead Strong Branching via a Stochastic Abstract Branching Model}
%
%
\author{Gioni Mexi\inst{1}\thanks{The first two authors contributed equally to the work presented in this paper.}\orcidID{0000-0003-0964-9802} \and
Somayeh Shamsi\inst{2}\repeatthanks{}\orcidID{0009-0002-0726-7173} \and
Mathieu Besançon\inst{1,3}\orcidID{0000-0002-6284-3033} \and
Pierre Le Bodic\inst{2}\orcidID{0000-0003-0842-9533}}

\authorrunning{Mexi, Shamsi et al.}
%
\institute{
Zuse Institute Berlin, Germany\\
\email{mexi@zib.de}
\and
Monash University, Australia\\
\email{\{somayeh.shamsi,pierre.lebodic\}@monash.edu }
\and
Université Grenoble Alpes, Inria, LIG, France\\
\email{mathieu.besancon@inria.fr}
}
\maketitle              
\begin{abstract}

Strong Branching (SB) is a cornerstone of all modern branching rules used in the Branch-and-Bound (BnB) algorithm, which is at the center of Mixed-Integer Programming solvers.
In its full form, SB evaluates all variables to branch on and then selects the one producing the best relaxation,
leading to small trees, but high runtimes.
State-of-the-art branching rules therefore use SB with working limits to achieve both small enough trees and short run times. So far, these working limits have been established empirically.
In this paper, we introduce a theoretical approach to guide how much SB to use at each node within the BnB.
We first define an abstract stochastic tree model of the BnB algorithm where the geometric mean dual gains of all variables follow a given probability distribution.
This model allows us to relate expected dual gains to tree sizes and explicitly compare the cost of sampling an additional SB candidate with the reward in expected tree size reduction.
We then leverage the insight from the abstract model to design a new stopping criterion for SB, which fits a distribution to the dual gains and, at each node, dynamically continues or interrupts SB.
This algorithm, which we refer to as Probabilistic Lookahead Strong Branching,
improves both the tree size and runtime over MIPLIB instances, providing evidence that the method not only changes the amount of SB, but allocates it better.

\end{abstract}
\section{Introduction}

%
%
%
%
Mixed-integer programming (MIP) is a powerful framework to model and solve optimization problems that include combinatorial structures.
It consists in minimizing a linear objective function over a polyhedron described by linear (in)equality constraints, where a subset of the decision variables must take integer values.

Despite the NP-hardness of general MIP solving, intensive research on the development of efficient algorithms grounded in polyhedral theory \cite{conforti2014integer,koch2022progress} and carefully designed combinations of these algorithms made generic solvers capable of optimizing to proven optimality problems that were previously considered out of reach.


The core of most MIP solvers is the Branch-and-Bound algorithm (BnB)~\cite{achterberg2007constraint}.
It recursively partitions the solution space of the MIP into subproblems and computes their continuous Linear Programming (LP) relaxations.
If the LP solution to a subproblem is integral, i.e., it is a solution to the MIP, then the best MIP solution is updated.
If the optimal value of the LP relaxation of a subproblem is not better than the current best integer solution known for the MIP, this subproblem can be pruned.
Otherwise, the subproblem is further divided.
We refer interested readers to \citet{conforti2014integer} for a recent overview on MIP and to \citet{achterberg2007constraint} for the algorithmic aspects of MIP solving.

Modern solvers obtained outstanding progress in speed and scale over the last decades from improvements in presolving, cutting planes, primal heuristics, and the algorithmic choices present at multiple phases of the solving process.
One crucial component is \emph{branching}, or variable selection, which consists, given a relaxation solution $\hat{\mathbf{x}}$, in finding a variable $x_j$ that violates its integrality constraint to partition the space with two subproblems
with constraints ${x}_j \leq \lfloor \hat{{x}}_j \rfloor$ and ${x}_j \geq \lceil \hat{{x}}_j \rceil$ respectively.
The set of integer variables that have a fractional value in the relaxation solution $\hat{\mathbf{x}}$ of a node are referred to as (branching) candidates.
Although branching on any fractional candidate ensures that BnB attains the MIP optimum, the choice of the candidate has a significant impact on how implicit the enumeration is, i.e. how fast the search is.
Indeed, \cite{achterberg2007constraint} reports that branching on an arbitrary candidate leads to trees 5.43 times larger on average than the best branching rule at the time, which leads to a slowdown factor of 2.26 compared to the state-of-the-art at that time.

One branching technique is \emph{strong branching} (SB), originally proposed in \cite{applegate1998solution}.
SB assesses each branching candidate by solving the two LPs resulting from the partition mentioned above.
The differences in optimal value of the two LPs compared to the current node are referred to as \emph{downgain}, \emph{upgain} when adding the bound constraint forcing the variable ``down'' or ``up'' with respect to its continuous relaxation value, respectively.
These gains in relaxation value or \emph{dual gains} are a key metric to evaluate the progress of the tree towards a proof of optimality.

Despite looking ``only'' one level of depth down the tree, SB is a robust predictor for candidates that lead to small tree sizes, relative to other tractable branching rules. SB can in particular, for some classes of instances, offer theoretical guarantees compared to the optimal tree size \cite{dey2023theoretical}.
In addition to its predictive power for variable selection, strong branching offers additional benefits in modern MIP solvers, including node pruning, primal solutions, and domain propagation.

The current default branching algorithm in SCIP is an extension of reliability~\cite{achterberg2005branching} and hybrid branching~\cite{achterberg2009hybrid}.
It uses in particular strong branching calls for variables that are considered \emph{unreliable} (i.e., which have not been branched on a certain number of times)
and then relies on a weighted score dominated by pseudocosts for reliable variables.
Even when all candidates are still unreliable -- e.g., at the root node -- evaluating
all of them would be prohibitively expensive, solving twice as many LPs as there are candidates.
If the branching rule calls strong branching to evaluate candidates, a
\emph{lookahead} strategy is used: if a certain number of
successive variables are evaluated without any improvement over the best candidate, the procedure stops and returns the current best.
This maximum lookahead value is essentially proportional to a fixed parameter $L$.
Conceptually, this problem can be viewed as a generalized secretary problem with interview cost \cite{gianini1976infinite,bartoszynski1978secretary} or a variant of Pandora's box problem~\cite{beyhaghi2023recent,weitzman1978optimal},
in which indistinct candidates with hidden rewards have to be evaluated sequentially to maximize a net difference between the selected candidate's reward and the total interview costs.
We will refer to the number of unsuccessful successive candidates as the lookahead value throughout this paper and will propose a novel stopping criterion equipped with a dynamic lookahead replacing the static $L$ value.

A probabilistic model of variable dual gains was proposed in \citet{hendel2015enhancing} in order to determine
whether a given variable is unreliable and to choose which one to select for branching.
More precisely, the dual gains of any variable are modeled as random variables following a normal distribution (i.i.d.~across all nodes), and a parametric statistical test is used to determine a probability that a variable can yield a greater gain than the incumbent.
This per-variable distribution can only provide guidance deep in the tree when a variable has been branched on several times to offer statistically significant inference.
In contrast, our probabilistic model represents the dual gains of all variables at the current node as drawn from a distribution,
making our approach applicable even when some variables have not been sampled yet, and therefore from the very first nodes of the tree.
Furthermore, our model can be adapted to any distribution for which the cumulative function can be evaluated, alleviating the need for restricting assumptions such as normality.

\emph{Contributions.}
We define an abstract model of the Branch-and-Bound algorithm, based on the Multi-Variable Branching (MVB) model of \citet{le2017abstract},
incorporating unknown variable gains akin to Pandora's box problem and a particular form of tree restart. We coin the new model as \newMVB{}, or PVB.
We determine additional assumptions under which strong branching should continue or stop and branch on the best candidate found so far, defining a stopping criterion for our Probabilistic Lookahead Strong Branching (PL-SB) algorithm.

We study the dual gain distributions at nodes of the SCIP BnB algorithm where SB is performed on MIPLIB instances and construct a mixed distribution which captures these observed samples well. Beyond our work, we anticipate that this finding will help future research MIP research.
We perform simulation runs on dual gains of MIP instances, comparing the total number of nodes obtained with SCIP current maximum lookahead criterion and PL-SB.
These first simulations show that PL-SB is capable of reducing the expected tree size in \newMVB{}, even when the distribution is not perfectly known but fitted to the dual gains observed so far.

Finally, we design a practical strong branching algorithm incorporating working limits within PL-SB,
and implement it within the SCIP default reliability branching rule.
We conduct extensive experiments to assess its reliable performance on established MIP benchmarks, yielding a reduction in both time and number of nodes, two results that could not be achieved by simply changing some parameters of the current SB strategy.

\section{Parameter Tuning on Strong Branching Yields no Improvements}\label{sec:parameter_tuning}

The reliability pseudocost branching technique used in SCIP has been engineered to perform well across an array of instances of various sizes, difficulties, and structures, designed with careful tradeoffs between predictive power and 
However, with the tremendous progress achieved in MIP solving in the last decades~\cite{koch2022progress} and the successive algorithmic changes,
one could assume that the default parameters of the branching algorithm do not achieve optimal performance anymore and that simply tuning the available parameters already achieves significant performance improvements.

In the default configuration of SCIP, strong branching is restricted by a maximum lookahead $L^{\max}$ defined as:
\begin{equation*}
    L^{\max} = (1 + c^{\text{uninit}} / c^{\text{all}}) \cdot L,
\end{equation*}
where $ c^{\text{uninit}} / c^{\text{all}}$ is the fraction of uninitialized branching candidates and $L$ is the lookahead parameter which by default is fixed to 9.
Moreover, the solver restricts the number of strong branching simplex iterations to:
\begin{equation*}
    \gamma^{\max} = \gamma^{\text{node}} + K,
\end{equation*}
where $\gamma^{\text{node}}$ is the total number of regular node LP simplex iterations and $K$ is the number of additional iterations in SB, which by default is fixed to $10^6$. 

The choice of values for $L$ and $K$ is based on rigorous empirical studies on diverse instance sets, and as we show in \Cref{tab:default_tuning} it is not straightforward to find alternative values that improve the overall performance of the solver.  
\Cref{tab:default_tuning} shows the results of preliminary experiments on the MIPLIB 2017 benchmark~\cite{gleixner2021miplib} with three different random seeds for different values for $L$ and $K$. To compare the running time and number of nodes of different settings we use the shifted geometric mean. 
The first row is SCIP with default parameter values and is compared with each of the three shown settings. We report on the quotient of running time for ``affected'' instances, i.e., instances on which the changed parameters affect the solver default behavior, and the quotient of running time and number of nodes for ``affected-solved'' instances, i.e., affected instances solved by both default SCIP and the setting that we compare it with.

Decreasing the default lookahead parameter $L$ from 9 to 7 leads to worse performance on all affected instances and on affected instances solved by both settings. This behavior can be explained by the fact that worse branching decisions are made because of the smaller maximum lookahead value. Similarly, increasing the lookahead parameter $L$ to 11 leads to worse overall performance, an effect that is not as obvious, since intuitively, better branching decisions are made.
However, SCIP limits the budget of simplex iterations allowed in strong branching. Therefore, it can occur that calling strong branching more frequently at the start of the solving process, e.g., by increasing the parameter $L$, may restrict the calls of strong branching later in the search.
To assess the effect of the two parameters, we conduct a final experiment. Here, we increase the budget of LP simplex iterations by setting $K=1.2\cdot10^6$. However, overall, the final setting also leads to performance deterioration compared to the default settings.

\setlength{\tabcolsep}{10pt}
\begin{table}[ht]
\caption{Evaluating SCIP performance for different values of the strong branching parameters $L$ and $K$.}
\label{tab:default_tuning}
\centering
\footnotesize
\begin{tabular}{lrrrrrrr}
    \toprule
      &\multicolumn{2}{c}{affected}  & \multicolumn{2}{c}{affected-solved} \\
    \cmidrule(lr){2-3}\cmidrule(lr){4-5}
    Setting & \# &time(\%)&  time(\%) &nodes(\%) \\
    \midrule
    $L=9, K=10^6$ (default) & - &  100.0 &100.0 &100.0 \\
    $L=7, K=10^6$& 278 & 104.3 &104.4 &104.8\\
    $L=11, K=10^6$& 269 & 102.1 &104.0 & 105.0\\
    $L=11, K=1.2\cdot10^6$& 271 & 102.4 &105.2&106.5\\
    \bottomrule
\end{tabular}
\end{table}

Finally, we mention that we also experimented with other parameters which also did not lead to conclusive trends or performance improvements.
Even though on average all settings from \Cref{tab:default_tuning} are worse than the default setting, the virtual best uses 16\% fewer nodes. This indicates that a dynamic decision of whether or not the solver should continue applying SB until the maximum lookahead parameter is reached could potentially improve the overall performance of the solver. This virtual best is obtained by applying a single rule throughout the tree, that potential gain could even become higher by applying a rule at each node where a SB decision is made.

\section{\newMVB{}: A Stochastic Abstract Model for Branch-and-Bound}

The preliminary experiment of \cref{sec:parameter_tuning} leads to the conclusion that the parameters controlling strong branching in SCIP are already tuned, but the virtual best configuration shows that adjusting the parameters on a per-instance basis could yield significant improvements.
This implies that a new algorithm that more adaptively drives strong branching in SCIP is necessary to better allocate computational resources.

Given the complexity of MIP solvers and the myriad of ``adaptive'' techniques that could be tapped into, there are virtually infinitely many algorithms that could be designed for this purpose.
Undoubtedly, many would work, in the sense that they would improve performance on reference benchmarks, but at the cost of increased complexity and degrees of freedom in the solver configuration space, without an understanding of why the method works.
We instead focus on building a theoretical model for the BnB process that is well-adapted to represent strong branching and use it as a theoretical foundation for an algorithm so that the improvements it incurs are well understood and can be transferred and generalized to other solvers, methods, and problem types.

Therefore, we start by constructing an abstract BnB model extending the approach of \citet{le2017abstract}, which has proven to capture important aspects of BnB well enough to improve performance in SCIP.
However, all the models proposed in \citet{le2017abstract} suppose that the dual gains are fixed throughout the tree and known in advance, which renders an algorithm like SB completely superfluous in this setting.

For this reason, we design in this section an extended abstract model in which the dual gains are fixed, but not known until the corresponding variable is branched on.
This is not true in general. However, MIP solvers internally estimated the dual gains of variables by a pair of values, the \emph{pseudocosts}. This assumption is useful to analyze the resulting abstract models. 
Based on this model, we design in \cref{sec:prob_lookahead_for_PMVB} a stopping criterion for SB based on the likelihood that additional strong branching iterations provide a better branching candidate, as measured by the tree size.

\subsection{Background on BnB Abstract Modeling}

The original paper on abstract tree models for BnB~\cite{le2017abstract} sets a foundation for studying the properties of BnB algorithms.
The approach focuses on proving a given bound (i.e.~closing a given dual gap) by branching on variables with fixed and known dual gains.
It provides a theoretical framework to better understand BnB and has proven useful for improving solver performance, providing a new rule for variable selection in branching, now a component of SCIP's BnB implementation since in SCIP 7~\cite{gamrath2020scip}.

In the abstract models of~\cite{le2017abstract}, a variable $x_j$ is represented by a pair of values $(l_j, r_j)$.
The value $l_j$ (resp. $r_j$) models the left (resp. right) dual gain of $x_j$, i.e. the change in objective value in the LP relaxation, of the left (resp. right) child when $x_j$ is branched on.
This means that, when $x_j$ is branched on at node $i$, the bound on the left child of $i$ is improved by $l_j$, and the bound on the right child of $i$ is improved by $r_j$.
The gains $(l_j, r_j)$ assigned to $x_j$ remain constant throughout the tree in the abstract model, regardless of the node depth or the dual bound value at that node. 
Given a dual gap $G$, a tree closes the gap $G$ (i.e., proves the bound $G$) if each leaf has a value of at least $G$.
\begin{definition}[Tree dual gap]
The dual gap $G_i$ closed at node $i$ is given by:
\begin{align*}
G_i =
\begin{cases}
    0 & \text{if $i$ is the root node,}  \\
    G_h + l_{j} & \text{if $i$ is the left child of node $h$,} \\
    G_h + r_{j} & \text{if $i$ is the right child of node $h$,}
\end{cases}
\end{align*}
where node $h$ is the parent of $i$, if there is one, $G_h$ is the gap at $h$, and $(l_j, r_j)$ is the pair of gains of the variable branched on at $h$.
\end{definition}


 
One of the problems defined using this abstract model is the Multiple Variable Branching (MVB) problem.
The MVB model is a simplified version of the branch and bound tree, where $n$ variables are given in input, and where one variable $x_j$ with gains $(l_j,r_j)$ must be branched on at each node to build a tree that closes a gap $G$.
In MVB, one variable can be branched on multiple times on a path from the root to a leaf.
However, because of the simplicity of MVB, some aspects of a modern implementation of a BnB, such as SB, cannot be modeled.
Indeed, in a BnB in which the dual gains of the variables would be fixed and known in advance as they are in MVB, information discovery from candidate sampling as performed by SB would not be relevant.
Therefore, in \cref{sec:new_model}, we introduce an abstract model in which the dual gains of the variables are unknown.

\subsection{An Abstract BnB Model with Stochastic Variable Gains}\label{sec:new_model}
We extend MVB to incorporate variable gains that are not known in advance.
They are still fixed throughout the tree, so the gains of a variable are fully revealed once the variable is branched on, including if the variable is evaluated by strong branching, even if it is not selected as a branching candidate.




We will represent SB in an abstract BnB model using \emph{restarts}, a common technique from constraint and mixed-integer programming.
We define a restart as the removal of all nodes from the tree, with two important properties:
\begin{inparaenum}[(a)]
\item the nodes evaluated before a restart are still accounted for in the total cost of the tree, and
\item the observed dual gains remain valid after a restart.
\end{inparaenum}
Performing SB therefore consists in branching on one variable, recording its dual gain, restarting the tree, and iterating.
We can now define the PVB problem on our abstract BnB model:
\begin{definition}[Pandora's Multi-Variable Branching Problem]
Given $n \in \zplus$ variables with random dual gains in $\rpluspower{2}$ that follow their own distribution, a target gap $G\in \rplus$, and $k \in \zplus$, can a BnB tree that closes the gap $G$ be built using at most $k$ nodes, using each variable as many times as needed, where the random dual gains of a variable are revealed when it is first branched on, and where restarts are allowed?
\end{definition}


Note that one can effectively simulate SB in PVB by restarting after branching on one variable, and finally build the final tree without restarting.


Note that PVB does not make any assumption on whether and how the dual gains of the variables are distributed and related.
In \cref{sec:SB_in_MVB}, we will assume that the dual gains are i.i.d.~and follow a distribution that has a defined cumulative distribution function.
\cref{sec:distribfitting} shows that this assumption is realistic by assessing several distribution families.

\section{Probabilistic Lookahead: a Stopping Criterion for Strong Branching on Pandora’s MVB}\label{sec:prob_lookahead_for_PMVB}
\cref{sec:parameter_tuning} shows the potential of more selectively using SB in MIP solvers.
In this section, we design a SB strategy for PVB, and we compare it to SCIP's default rule on PVB instances. 
Simulation experiments on instances of the new abstract model show that the new criterion is substantially better in total LP evaluations than the fixed maximum lookahead stopping criterion used in SCIP.

\subsection{Necessary Assumptions}
The smallest tree that closes a gap $G$ by repeatedly branching on a single variable is referred to as \emph{SVB tree}, in line with~\cite{le2017abstract}. 
In order to establish a stopping criterion that is grounded in theory, we first make a number of explicit assumptions:
\begin{enumerate}[label=\textbf{A\arabic*}, start=0]
    \item The minimum size MVB tree is the SVB tree of the best variable. \label{assumption:small_MVB_is_SVB}
    \item The size of an SVB tree that closes $G$ by only branching on variable with gains $(g_j, g_j)$ is equal to the size of the tree that only branches on a variable with gains $(l_j, r_j)$, where $g_j=\sqrt{l_j r_j}$ is the geometric mean of dual gains of variable $x_j$.\label{assumption:geomean_ts_approx}
    \item The geometric mean of the dual gain of all variables follows the same distribution.\label{assumption:geomean_distrib}
\end{enumerate}

Assumption~\ref{assumption:small_MVB_is_SVB} does not hold in general, and simple counterexamples can be constructed~\cite{anderson2021further, le2017abstract}, but a key result from~\cite[Theorem 7]{le2017abstract} is that, for large gaps, the size of an optimal MVB tree grows at the same rate as the size of the best SVB tree.
Therefore, in this paper, we use the SVB tree size as a measure of the quality of a variable.
Assumption~\ref{assumption:geomean_ts_approx} is perhaps the furthest away from actual BnB trees, based on early experiments. 
However, the numerical improvements shown in \cref{sec:experiments_on_MVB} and \ref{sec:experiments_on_MIP} show that this approximation is close enough to be representative of the quality of the variables.
We show that Assumption~\ref{assumption:geomean_distrib} is close to what can be observed on dual gains of real instances separately in \cref{sec:distribfitting}, as these results are interesting in and of themselves.

We focus on the geometric mean of a variable because its square, the \emph{product}, is a state-of-the-art scoring function for variables \cite{achterberg2007constraint}. 
Hence, it has been shown that it is a good practical heuristic for ranking branching candidates.
We now use the assumptions above to design a stopping criterion for SB.

\subsection{Probabilistic Lookahead: a Stopping Criterion for Strong Branching on \newMVB{}}\label{sec:SB_in_MVB}
We now use \newMVB{} to design a new stopping criterion for SB, which we call \emph{Probabilistic Lookahead}.
At each SB iteration $i$, the tree is restarted, and an arbitrary variable $x_j$ is branched on, which reveals its dual gains $(l_j,r_j)$.
We denote $S_i$ the set of variables that have known dual gains at iteration $i$.
The stopping criterion then decides whether to further iterate and expand $S_{i+1}$ or stop and use a variable in $S_i$.
Using Assumption~\ref{assumption:geomean_ts_approx}, we estimate the SVB tree size as $2^{d_j + 1} -1$, where $d_j = \lceil\frac{G}{g_j}\rceil$, which is the exact SVB tree size of variable $(g_j, g_j)$.
Therefore, according to this estimate, the best variables found so far are those that minimize $\dmin{i}= \min_{x_j \in S_i} d_j$ or equivalently maximize the mean dual gain.

In the following statement and throughout, we use the term ``size'' when referring to the nodes of a tree, and we write ``nodes used'' to refer to all nodes created so far, including those that have been discarded by a restart.

\begin{proposition}
If we stop SB at the end of iteration $i>0$ with a depth $\dmin{i}$, the minimum number of nodes used to construct the MVB tree is:
\begin{equation*}
t_i(G) = 2^{\dmin{i} + 1} - 1 + 2 \cdot i.
\end{equation*}
\end{proposition}
\begin{proof}
First, according to Assumption~\ref{assumption:geomean_ts_approx}, the total size of the best MVB tree is the size of the best SVB variable, therefore we can suppose without loss of optimality that the same variable is branched on throughout the tree.
Second, we can assume that the variable branched on is in $S_i$, as otherwise fewer nodes are required by stopping SB at $i=0$,
i.e.,~branching on a variable on which we have no information.
Therefore, the minimum-size tree is a perfect binary tree, where all leaves have depth $\dmin{i}$, and whose size is thus $2^{\dmin{i} + 1} - 1$.
Accounting for the $i$ SB iterations, the total number of nodes used is $2^{\dmin{i} + 1} - 1 + 2 \cdot i$. \qed
\end{proof}

In order to decide whether to stop, we need to estimate the size of the tree if we continue iterating SB.
Following Assumption~\ref{assumption:geomean_distrib}, we suppose that we can compute the probability that the next variable will decrease the depth required,
i.e.~that $\dmin{i+1} < \dmin{i}$.
For $d=\{1, \dots, \dmin{i}\}$, we denote $p_d$ the probability that $\dmin{i+1}=d$.




\begin{theorem}
The expected number of nodes used if we stop SB at the end of iteration $i+1$ is:
\begin{align}
E[t_{i+1}(G)] = \sum_{d=1} ^ {\dmin{i}} ((2 ^ {d + 1} -1) \cdot p_d) + 2 \cdot (i+1).\label{eq:sb}
\end{align}  
\end{theorem}

\begin{proof}
At iteration $i+1$, the minimum depth of the tree is $d = 1$, and the maximum depth is $\dmin{i}$.
Hence $\sum_{d=1} ^ {\dmin{i}} p_d = 1$.
Therefore, the expected value of the size of the final tree, if SB stops at iteration $i+1$, is $\sum_{d=1}^{\dmin{i}}(2^{d + 1} -1) \cdot p_d$.
Additionally counting the $i+1$ strong branching calls, each of them producing two nodes, we obtain~\eqref{eq:sb}.\qed
\end{proof}

The Probabilistic Lookahead stopping criterion then consists in comparing $t_i(G)$ and $E[t_{i+1}(G)]$ at the end of every iteration $i$, to decide whether to continue SB or to stop it and branch on the best candidate found so far.

\subsection{Evaluating the Probabilistic Lookahead on \newMVB{}}\label{sec:experiments_on_MVB}

In this section, we investigate the effectiveness of the Probabilistic Lookahead SB stopping criterion on instances of \newMVB{}.
In order to create realistic instances, we extract dual gains observed during the solution of MIP instances.
Specifically, we use the \emph{full} SB rule of SCIP, with all other parameters at default values, i.e.~we perform SB on all branching candidates, and exit after the root node.
Together with a gap $G$, all dual gains collected for one MIP instance make a PVB instance.

The experiment compares different stopping criteria when applied to \newMVB{} in terms of total number of nodes.
The baseline is the state-of-the-art lookahead strategy (see \cref{sec:parameter_tuning}), which is the default in SCIP and which we refer to as ``Fixed Lookahead''.
We compare this to the Probabilistic Lookahead SB, with different dynamically-fitted distributions: mixed exponential and mixed Pareto, as they are the best two found in \cref{sec:distribfitting}, and the exponential to assess the impact of representing the mass point at 0.

Because we observed that the performance of these stopping criteria is similar relative to the others on all instances, here we only present results on the \texttt{Trento1} MIPLIB instance.
For each gap value $G$, we conduct 1000 experiments.
The average overall number of nodes and number of nodes used for SB are presented in \cref{table:sb-timtab}.

The outcome demonstrates that the abstract model performs significantly better than the SCIP strategy in terms of number of nodes required to close the gap, across all examined cases.

\begin{table}[ht]
\centering
\caption{Comparison of the number of nodes with different stopping criteria for Pandora's MVB problem based on dual gains from the \texttt{Trento1} instance. The number of nodes produced by SB is given in parentheses.}
\label{table:sb-timtab}
\begin{tabular}{cccccc}
    \toprule
    \multirow{2}{*}{Gap} & &\multicolumn{3}{c}{Fitted distribution} & \\
    \cmidrule(lr){3-5}
    & Fixed lookahead & Exp. & Mixed Exp. & Mixed Pareto & Full SB \\ 
    \midrule
    1000  & 112 (64) & 91 (42) & \textbf{90} (44) & \textbf{90} (44) & 1009 (966)\\
    2000  & 161 (64) & 137 (42) & \textbf{136} (44) & 145 (52) & 1051 (966) \\
    3000  & 210 (66) & 189 (44) & \textbf{185} (46) & 194 (62) & 1093 (966) \\
    4000  & 338 (66) & \textbf{235} (44) & 236 (46) & 259 (68) & 1133 (966)\\
    5000  & 1267 (64) & 306 (48) & \textbf{302} (50) & 348 (96) & 1175 (966)\\
    6000  & 1810 (64) & 370 (52) & \textbf{365} (52) & 387 (110) & 1217 (966) \\
    6500  & 26258 (64) & 462 (56) & 415 (54) & \textbf{409} (116) & 1237 (966) \\
    \bottomrule
\end{tabular}
\end{table}

\section{The Geometric Mean of Dual Gains Follows a Mixed Distribution}\label{sec:distribfitting}
Before presenting the algorithm we design for strong branching, we assess the accuracy of Assumption~\ref{assumption:geomean_distrib}, namely, that there is a probability distribution that captures well the geometric mean of dual gains.
In practice, the dual bound changes when branching on a given variable depend heavily on the particular structure of the problem.
Fixing decision variables that impact the rest of the problem structure
-- e.g., design variables in a network design problem -- will result in structural changes
that are likely to change the relaxation bound a lot more than variables that have a more ``local'' influence on 
-- e.g., routing variables in the same network design example.
We compare the distribution of mean dual gains for the MIPLIB 2017 benchmark instances at the root node with different distributions fitted with their maximum likelihood estimators.
Three requirements for the distribution we fit are having a zero probability on negative gains, allowing for a mass point at zero, and having a support that is a proper superset of the interval $\left[0,g_{\max}\right]$, since we want to estimate the probability of a new sample having a value greater than the current maximum mean dual gain $g_{\max}$, which would be zero if the support is equal to the interval, e.g., with the uniform or empirical cumulative distributions. 
The reason for the mass point at zero is the dual gain often being zero for many variables, i.e., the LP optimum after branching remains on the original optimal face.
To accommodate this characteristic, we model dual gains with a mixed continuous-discrete distribution:
\begin{align*}
P[G \leq g] = p_0 + (1 - p_0) F_D(g; \theta),
\end{align*}
with $p_0$ the probability of a zero dual gain and $F_D(\cdot\ ; \theta)$ the cumulative distribution function of a continuous probability distribution $D$ with parameters $\theta$.
We tested several continuous distribution candidates including the uniform, log-normal, exponential, Pareto, and normal distributions.
Note that the normal distribution was used mostly as a ``control'', since its nonempty support on negative numbers makes it a poor modeling choice.
We fit the aforementioned distributions only on nonzero gains and evaluate their fitness with the Kolmogorov-Smirnov non-parametric test implemented in \texttt{Distributions.jl}~\cite{JSSv098i16}.
This test evaluates the maximum distance between the empirical cumulative distribution function of the data and the cumulative function of the tested distribution (with parameters estimated from the data).
The results hint that the exponential and Pareto distributions appear as good candidates, based on a proportion of dual gains series on which the null hypothesis fails to be rejected (20\% of the dual gains series for Pareto, 33\% for exponential, when evaluating series with at least 10 dual gains, and rejecting the null hypothesis for a $p$ value of $5\%$), although we notice a difference between the fitness of the distribution and its performance on Pandora's MVB and within the actual solving process of a solver in the next section. Future work will investigate this relationship and gap between the predictive capabilities of a distribution and its performance when used within the probabilistic lookahead algorithm.


\section{Probabilistic Lookahead within SCIP's Strong Branching}
Based on the abstract model for strong branching and a probability distribution fitted on observed dual gains, we can compute a probability that sampling a new strong branching candidate attains a higher mean dual gain than the current one, and correspondingly, reduces the total number of LPs to evaluate.
Equipped with this estimation, we can dynamically decide whether to sample a new candidate or stop strong branching, and then branch on the candidate offering the highest predicted score, either with the dual gain sampled by strong branching or with the pseudocosts for the other reliable candidates.
We summarize our method in Algorithm~\ref{alg:strongbranchinit}.
One central difference of the algorithm integrated within reliability pseudocost branching compared to the abstract model presented in \cref{sec:new_model} is
that it uses a weighed sum as a score function, aggregating not only the dual gains (or pseudocosts for reliable candidates), but also a range of other criteria used for branching,
such as inference score and conflict analysis~\cite{achterberg2009hybrid}, the fraction of the times a variable led to node cutoff, the score of cuts corresponding to the variable~\cite{turner2023branching}, or dual degeneracy~\cite{gamrath2020exploratory}.
The weights of these scores however remain orders of magnitude lower than pseudocosts or dual gains, effectively making them act as tie-breakers, we thus consider that our model -- although ignoring these additional aspects -- captures the essential information.
\begin{algorithm}[ht]
\caption{Probabilistic Lookahead Strong Branching}
\label{alg:strongbranchinit}
\SetNlSty{textbf}{(}{)}
\SetKwInOut{Input}{Input}
\SetKwInOut{Output}{Output}
\SetArgSty{textnormal}
\Input{Set of fractional candidates, $\phi = 0.6$}
\Output{Updated pseudocosts, best $k$ among unreliable candidates}
    $L^{\max} \gets (1 + c^{\text{uninit}} / c^{\text{all}}) \cdot L$\;
    $    \gamma^{\max} = \gamma^{\text{node}} + K$\;
\For{$x_j$ \textbf{in} at most $C$ fractional unreliable candidates}{
    Apply SB on $x_j$ to compute the down- and upgain $\Delta^-$, $\Delta^+$\;
    Update the pseudocosts with the gains $\Delta^-$, $\Delta^+$\;
    $g_j \gets \sqrt{(\Delta^- + \epsilon) \cdot (\Delta^+ + \epsilon)} - \epsilon$\;
    Update the distribution with the new geometric mean gain $g_j$\;
    $s_j \gets \text{score}(g_j)$\;
    $k \gets \argmax_{k\in 1..j} s_k$\;
    \If{${s_k}$ has not changed for $L^{\max}$ iterations}{
        \textbf{break}\;
    }
    \If{strong branching LP iterations exceed the limit $\gamma^{\max}$}{
        \textbf{break}\;
    }
    \If{$s_k$ not changed for $\phi \cdot L^{\max}$ iterations \textbf{and} enough nonzero samples}{
        Test the expected tree size now and with one more SB iteration\;
        \If{no expected improvement}{
            \textbf{break}\;
        }
    }
}
\end{algorithm}
The small perturbation $\epsilon$ is used to produce shifted geometric mean values $g_j$ that can be compared, even when one of the down- or upgain is close to zero.
We note that when applying strong branching, there is a limit on the number of LP iterations per candidate, mitigating pathological cases where branching yields a high number of simplex iterations.

\section{Computational Experiments}\label{sec:experiments_on_MIP}

We implement our algorithm as a branching rule in SCIP 8~\cite{bestuzheva2023enabling} generalizing the hybrid branching algorithm~\cite{achterberg2009hybrid}.
The modified branching rule used for the computational experiments is available and open-source\footnote{The branching rule is integrated in \href{https://github.com/matbesancon/scip/tree/strongbranching\_dynamiclookahead\_2}{github.com/matbesancon/scip} on the 
\texttt{strongbranching\_dynamiclookahead\_2} branch.}
and will be integrated in the next major release of SCIP.

In our computational experiments, in order to lessen the effects of performance variability~\cite{lodi2013performance} and obtain accurate comparisons, we use a large test set of diverse problem instances.
In particular, our test set consists of the official benchmark set MIPLIB 2017~\cite{gleixner2021miplib}, with random seeds 0 to 4, amounting to 1200 instance-seed pairs. For the remainder of this paper, we will refer to each instance-seed pair as an \textit{instance}.

All experiments are run in exclusive mode (one job per machine) on a cluster equipped with Intel Xeon Gold 5122 CPUs running at 3.60GHz, where each run is restricted to a 4-hour time limit and 48GB of memory.
After preliminary experiments, we observed that the mixed Pareto distribution not only fits the data as observed in \cref{sec:distribfitting}, but also works better than the exponential distribution, which is the other one we assessed our method with.
Our method can only stop strong branching when enough non-zero samples are collected, and we reached 60\% of the maximum lookahead.

\Cref{tbl:1} compares default SCIP (``default'') to SCIP using our PL-SB strategy (``dynamic'').
Besides the number of instances solved to optimality, the main measure we are interested in is the shifted geometric mean of running time and number of BnB nodes.
The shift for the running time is set to 1 second and for the nodes to 100.
We compare the two settings on instances where the solver behavior is affected by our dynamic strong branching method (``affected''), and also on the whole test set (``all''). Moreover, we show results for the subcategories with solved instances by both settings  (``$*$-solved''), and also harder instances solved in more than 1000 seconds by at least one of the two settings (``$*$-solved-1000+''). In the last two columns, we show for each category the relative difference of the two settings as a percentage. We highlight relative differences exceeding 1\% by presenting them in bold in the table.

On the entire test set (``all''), we solve ten more instances with the dynamic SB strategy and observe a substantial improvement in both runtime and number of nodes on instances affected by our method.
On ``affected'', there is a decrease in both running time and nodes. In particular, for harder instances, we observe a 8\% decrease in running time and a 9\% decrease in the number of nodes.
On the entire test set, we observe a slight improvement of 1-2\% in terms of both running time and number of nodes. This marginal impact is due to the limited number of instances affected by the PL-SB strategy. Our method stops strong branching earlier than the lookahead value in 20\% of the calls.
 
\begin{table}
\caption{Performance comparison of SCIP (default) and SCIP using the PL-SB strategy (dynamic)}
\label{tbl:1}
\scriptsize
\setlength{\tabcolsep}{7pt}
\begin{tabular*}{\textwidth}{@{}l@{\;\;\extracolsep{\fill}}rrrrrrrrr@{}}
\toprule
&           & \multicolumn{3}{c}{default} & \multicolumn{3}{c}{dynamic} & \multicolumn{2}{c}{relative} \\
\cmidrule{3-5} \cmidrule{6-8} \cmidrule{9-10}
Subset                & instances &                                   solved &       time &        nodes &                                   solved &       time &        nodes &       time &        nodes \\
\midrule
affected &  235 & 223 & 944.6 & 11019 & \textbf{233} & 907.1 & 10423 & \textbf{0.96} & \textbf{0.94} \\
affected-solved & 221 & 221 & 795.0 & 9596 & 221 & 771.7 & 9105 & \textbf{0.97} & \textbf{0.95} \\
affected-solved-1000+& 110 & 110 & 4062.7 & 42394 & 110 & 3756.5 & 38712 & \textbf{0.92} & \textbf{0.91} \\
\midrule
all & 1200 & 685 & 1591.9 & 6317 & \textbf{695} & 1578.9 & 6225 & 0.99 & \textbf{0.98}\\
all-solved & 683 & 683 & 302.2 & 3254 & 683 & 299.0 & 3198 & 0.99 & \textbf{0.98}\\
all-solved-1000+ & 254 & 254 & 3225.4 & 30179 & 254 & 3117.2 & 29013 & \textbf{0.97} & \textbf{0.96}\\
\bottomrule
\end{tabular*}
\end{table}

The results highlight that the working limits on SB (LP iterations and maximum lookahead) significantly impact both tree size and runtime.
We therefore perform an additional experiment with a 2-hour time limit, 20\% more SB simplex iterations (amounting to 120000 iterations) and a lookeahead of 11 instead of 9;
we then compare the static lookahead against our dynamic algorithm with this new setting in \Cref{tbl:2}.
The time limit was reduced compared to the first experiment to mitigate the computational burden of the overall experiments.

\begin{table}[ht]
\caption{Performance comparison of SCIP (default) and SCIP using the PL-SB strategy (dynamic) after increasing the maximum lookahead $L$ to 11 and the maximum number of simplex LP iterations $K$ in SB to $1.2\cdot 10^6$.}
\label{tbl:2}
\scriptsize
\setlength{\tabcolsep}{7pt}
\begin{tabular*}{\textwidth}{@{}l@{\;\;\extracolsep{\fill}}rrrrrrrrr@{}}
\toprule
&           & \multicolumn{3}{c}{default} & \multicolumn{3}{c}{dynamic} & \multicolumn{2}{c}{relative} \\
\cmidrule{3-5} \cmidrule{6-8} \cmidrule{9-10}
Subset                & instances &                                   solved &       time &        nodes &                                   solved &       time &        nodes &       time &        nodes \\
\midrule
affected&  188 & 179 & 577.4 & 5588 & \textbf{183} & 546.9 & 5281 & \textbf{0.95} & \textbf{0.94} \\
affected-solved & 174 & 174 & 479.8 & 4675 & 174 & 450.5 & 4368 & \textbf{0.94} & \textbf{0.93} \\
affected-solved-1000+& 74 & 74 &2644.3 & 19718& 74 & 2452.1 & 17619& \textbf{0.93} & \textbf{0.89} \\
\midrule
all & 1200 & 632 & 1168.3&4570& \textbf{636} & 1164.3 & 4541 & 1.00 & 0.99\\
all-solved & 627 & 627 & 225.4&2605 & 627 & 223.6&2566 & 0.99 & \textbf{0.98}\\
all-solved-1000+ & 201 & 201 & 2266.3&23660 & 201 & 2239.8 & 22993 & 0.99 & \textbf{0.97}\\
\bottomrule
\end{tabular*}
\end{table}

In this setting, we solve four more instances with the PL-SB algorithm and observe a substantial improvement in both runtime and number of nodes on affected instances solved by both methods.
In particular, we obtain 11\% fewer nodes on hard affected instances, and a 7\% smaller runtime.
Overall, we observe a 5\% improvement on affected instances.
We however note that this changed setting is not performing as well as the PL-SB algorithm with default SB parameters presented in \cref{tbl:1}, when compared on the basis of the 2-hour time limit.
The key conclusion we draw is that the dynamic lookahead mechanism improves over its static counterpart for different values of the SB parameters.

SCIP is also an exact global solver for Mixed-Integer Nonlinear Programs (MINLP), which remain much more challenging than their linear counterparts.
A lot of algorithmic questions considered answered for MIPs still need further investigation for MINLPs, see \cite{vigerske2018scip,bestuzheva2023enabling} for an overview of the formulation, methods, and implementation involved.
In particular, some techniques that are essential to the performance of MIP solving may be insignificant or detrimental on MINLPs; motivating us to assess the proposed Probabilistic Lookahead Strong Branching on MINLPs.

We test our method against the default lookahead on the \texttt{minlpdev-solvable}\footnote{The precise subset used will be made available in the companion repository of our paper.} subset of the MINLPLIB~\cite{bussieck2003minlplib,vigerske2021minlplib} which is the standard to assess the performance of SCIP on MINLPs.
We run the 169 instances with 5 seeds, resulting in 845 instance-seed pairs.
We exclude any instance-seed pair leading to numerical errors on any of the settings, leaving 817 instance-seed pairs.
The time limit for this experiment is one hour since this captures the bulk of the instances already (808 out of 817 instances).
On affected instances, our method produces modest improvements, 2\% in time and 4\% in the number of nodes.

The results are presented in \Cref{tbl:minlp}, highlighting a performance improvement consistent with that observed on MIPs.
Unlike on MIP instances, using the exponential instead of the Pareto distribution produces slightly better results, with a 3\% reduction in time and 4\% reduction in the number of nodes.
This experiment highlights that the new algorithm improves the performance of SCIP on a diverse set of MINLP instances with the same effect of fewer nodes and reduced time that we observed on MIPs.
It also suggests that a single distribution might not be the best to capture dual gains in all situations, opening the question of adaptively selecting a distribution to fit based on instance attributes.

\begin{table}
\caption{Performance comparison of SCIP (default) and SCIP using the PL-SB strategy (dynamic) on MINLP instances}
\label{tbl:minlp}
\scriptsize

\begin{tabular*}{\textwidth}{@{}l@{\;\;\extracolsep{\fill}}rrrrrrrrr@{}}
\toprule
&           & \multicolumn{3}{c}{default} & \multicolumn{3}{c}{dynamic} & \multicolumn{2}{c}{relative} \\
\cmidrule{3-5} \cmidrule{6-8} \cmidrule{9-10}
Subset                & instances &                                   solved &       time &        nodes &                                   solved &       time &        nodes &       time &        nodes \\
\midrule
affected             &    140 &                                      140 &    47.2 &         1604 &                                      140 &       46.5 &         1542 &        \textbf{0.98} &          \textbf{0.96} \\
\midrule
all            &    817 &                                      808 &       19.1 &         2429 &                                      808 &       19.0 &         2413 &        0.99 &          0.99 \\
all-solved           &      808 &                                      808 &       18.0 &         2273 &                                      808 &       17.9 &         2258 &   0.99 &          0.99 \\
\bottomrule
\end{tabular*}
\end{table}

\section{Conclusion}

In this paper, we defined Pandora's Multi-Variable Branching, an abstract model of the Branch-and-Bound algorithm that represents the dual gains at a particular node as a random variable with a known probability distribution.
All branching candidates represent one realization of this random variable, strong branching can then be viewed as sampling this distribution to obtain a candidate with the highest possible dual gain.
Equipped with this model and a family of probability distributions that model the dual gains observed on real instances, we explicitly estimate the expected reduction in the number of nodes from one more strong branching call,
compare that reduction with the cost of solving the additional LPs from strong branching, and dynamically decide after each candidate evaluation whether to continue or stop strong branching.
Computational experiments show that our Probabilistic Lookeahead Strong Branching algorithm improves over the static lookahead rule on the MIPLIB benchmark both in time and in number of nodes, especially on hard instances, a result that could not be achieved by simply adjusting SB parameters.

The different contributions of our paper open promising venues for future research.
The proposed PVB represents a variable as a single geometric mean gain, allowing us to build our lookahead strategy on the mean gains only.
Future work will include designing abstract models representing unbalanced binary trees, i.e.~with left and right dual gains.
The proposed algorithm performs better than the default rule on average on the benchmark set but does not dominate it on all instances.
Future work should investigate the criteria influencing the relative effectiveness of the two algorithms, and whether it is possible to improve over both of them based on static information available before starting to branch.
The presented model is oblivious to variable histories and only fits the distribution based on the dual gains observed at the current node.
An enhancement of the method could consider a Bayesian-like update of the fitted distribution, allowing the use of the dynamic lookahead earlier and with a better distribution fit.
We studied distributions mixing a mass point at 0 and a continuous distribution, which captures a good fraction of the series of dual gains at any given node.
This is to the best of our knowledge the first study modeling dual gains directly with a mixed distribution which appears essential from the data since many variables will yield a zero dual gain.
This probabilistic model has promising applications in other parts of reliability branching, e.g., to improve the estimation of pseudocosts and to determine when pseudocosts are reliable already.

\subsubsection{Acknowledgements} 

Research reported in this paper was partially supported through the Research Campus Modal funded by the German Federal Ministry of Education and Research (fund numbers 05M14ZAM, 05M20ZBM).
We thank Nicolas Gast and Bruno Gaujal for discussions on online decision-making including the Pandora's box problem.

\bibliographystyle{abbrvnat}
\bibliography{refs}

\end{document}